\documentclass{article}
\usepackage[a4paper]{geometry}
\usepackage{amsmath,amssymb,amsthm,enumitem,environ, color}
\usepackage[initials]{amsrefs}
    
\newtheorem{theorem}{Theorem}[section]
\newtheorem{lemma}[theorem]{Lemma}
\newtheorem{corollary}[theorem]{Corollary}

\numberwithin{equation}{section}

\renewcommand{\geq}{\geqslant}
\renewcommand{\leq}{\leqslant}

\setlist[enumerate]{leftmargin=20pt,itemsep=0pt,topsep=0pt}
\setlist[enumerate,1]{label=\emph{(\roman*)},ref={(\roman*)}}

\makeatletter
\renewcommand\section{\@startsection {section}{1}{\z@}%
                                   {-3.5ex \@plus -1ex \@minus -.2ex}%
                                   {1.3ex \@plus.2ex}%
                                   {\normalfont\large\scshape}}
\makeatother

\let\Im\relax
\DeclareMathOperator{\Im}{Im}
\let\Re\relax
\DeclareMathOperator{\Re}{Re}
\DeclareMathOperator{\degree}{deg}

\title{ \vspace{-5ex}\bf \Large A hyperbolic-distance inequality for holomorphic maps}

\makeatletter
\renewcommand*{\@fnsymbol}[1]{\hspace*{-10pt}}
\makeatother

\author{Argyrios Christodoulou and Ian Short\thanks{2010 Mathematics Subject Classification: Primary 30F45; Secondary 30C80.}\thanks{Key words: hyperbolic metric, holomorphic map, hyperbolic Riemann surface.}\thanks{School of Mathematics and Statistics, The Open University, Milton Keynes, MK7 6AA, United Kingdom.}}

\date{\vspace{-5ex}}

\begin{document}

\maketitle

\begin{abstract}
We prove an inequality which quantifies the idea that a holomorphic self-map of the disc that perturbs two points is close to the identity function.  
\end{abstract}

\section{Introduction}

The principal objective of this paper is to prove the following theorem, in which $\mathbb{D}$ denotes the open unit disc in the complex plane with hyperbolic metric $\rho$.

\begin{theorem}\label{cai}
Suppose that $f$ is a holomorphic self-map of $\mathbb{D}$ and $a,b,z\in\mathbb{D}$, with $a\neq b$. Then
\[
\rho(f(z),z) \leq K\big(\rho(f(a),a)+\rho(f(b),b)\big),
\]
where
\[
K= \frac{\exp\left(\rho(z,a)+\rho(a,b)+\rho(b,z)\right)}{\rho(a,b)}.
\]
\end{theorem}

A somewhat similar result was obtained for M\"obius transformations acting on the extended complex plane, using the chordal metric, in \cite{BeSh2010}.

A strength of Theorem~\ref{cai} is that the constant $K$ is independent of the function $f$. As a consequence, we can use the theorem to prove quantitative versions of existing results about holomorphic maps close to the identity function. For instance, it is known that if $(f_n)$ is a sequence of holomorphic self-maps of $\mathbb{D}$ such that $f_n(a)\to a$ and $f_n(b)\to b$, for two distinct points $a$ and $b$, then $(f_n)$ converges locally uniformly on $\mathbb{D}$ to the identity function. This can be proven by a normal families argument (see, for example, \cite{Sc1993}*{Theorem 2.4.2}), but such an argument does not give an explicit rate of convergence. In contrast, Theorem~\ref{cai} provides a rate of convergence, which allows us to make stronger statements; for example, the theorem tells us that  if the sum $\sum \rho(f_n(z),z)$ converges for $z=a,b$, then in fact it converges for \emph{any} $z\in\mathbb{D}$. 

If we fix two distinct points $a$ and $b$ in $\mathbb{D}$, then Theorem~\ref{cai} (and the inequality $\rho(z,b)\leq \rho(z,a)+\rho(a,b)$) show us that there is a constant $k$ depending only on $a$ and $b$ for which
\begin{equation}\label{keu}
\rho(f(z),z) \leq ke^{2\rho(z,a)}\big(\rho(f(a),a)+\rho(f(b),b)\big),
\end{equation}
for all $z\in\mathbb{D}$ and all holomorphic self-maps $f$ of $\mathbb{D}$. We now describe an example to show that the expression $e^{2\rho(z,a)}$ in this inequality cannot be  reduced significantly. 

For this example, we switch from $\mathbb{D}$ to the right half-plane model of the hyperbolic plane, denoted by $\mathbb{K}$, and consider holomorphic self-maps of $\mathbb{K}$. We continue to denote the hyperbolic metric by $\rho$, on $\mathbb{K}$ as on $\mathbb{D}$, and we make use of  the formula $\rho(u,v)=\log (v/u)$, for points $u$ and $v$ on the positive real axis, with $u<v$. 

 Let $a=1$ and $b=2$. Let $f_n(w)=w+1/n^2$ and $z_n=1/n$, for $n=1,2,\dotsc$. Then $e^{\rho(z_n,a)}=n$, and
\[
\rho(f_n(z_n),z_n) \sim\frac{1}{n},   \quad    \rho(f_n(a),a) \sim \frac{1}{n^2}
\]
(where, for two positive sequences $(x_n)$ and $(y_n)$, we write $x_n\sim y_n$ to mean that there is a positive constant $\lambda$ such that $x_n/\lambda < y_n <\lambda x_n$, for $n=1,2,\ldots$). So
\[
\frac{\rho(f_n(z_n),z_n)}{\rho(f_n(a),a)} \sim e^{\rho(z_n,a)}.
\]
That is, the quotient of the distortion of $f_n$ at $z_n$ by the distortion of $f_n$ at $a$ grows exponentially with the hyperbolic distance between $z_n$ and $a$. This examples indicates that the expression $e^{2\rho(z,a)}$ in inequality~\eqref{keu} cannot be made any smaller than $e^{\rho(z,a)}$.

The proof of Theorem~\ref{cai} uses the fact that any holomorphic self-map of $\mathbb{D}$ contracts the hyperbolic metric on $\mathbb{D}$, in the sense that if $f$ is such a map, then $\rho(f(z),f(w))\leq \rho(z,w)$, for all $z,w\in\mathbb{C}$, by the Schwarz--Pick Lemma. We observe, however, that the theorem fails for the class of contractions of $\mathbb{D}$ with the hyperbolic metric, which is broader than the class of holomorphic self-maps of $\mathbb{D}$. To see this, consider the function $f(w)=\Re w$, which contracts the hyperbolic metric on $\mathbb{D}$. Given any two distinct real numbers  $a$ and $b$ in $\mathbb{D}$, we have $\rho(f(a),a)=\rho(f(b),b)=0$, but $\rho(f(z),z)$ is positive for nonreal points $z$, so the inequality in Theorem~\ref{cai} fails. 

To illuminate later work, we record the following minor generalisation of Theorem~\ref{cai}. The statement of this result features a conformal automorphism of $\mathbb{D}$, which, by definition, is a bijective holomorphic map from $\mathbb{D}$ onto itself. Such maps are hyperbolic isometries, and using this property we can immediately deduce the result from Theorem~\ref{cai}.

\begin{corollary}\label{xjb}
Suppose that $f$ is a holomorphic self-map of $\mathbb{D}$, $h$ is a conformal automorphism of $\mathbb{D}$, and $a,b,z\in\mathbb{D}$, with $a\neq b$. Then
\[
\rho(f(z),h(z)) \leq K\big(\rho(f(a),h(a))+\rho(f(b),h(b))\big),
\]
where $K=\rho(a,b)^{-1}\exp(\rho(z,a)+\rho(a,b)+\rho(b,z))$.
\end{corollary}

Theorem~\ref{cai} and Corollary~\ref{xjb} could of course be stated with any simply connected hyperbolic Riemann surface in place of $\mathbb{D}$. Consider, now, a general hyperbolic Riemann surface $S$. If $S$ has a nonabelian fundamental group, then, in the space of holomorphic self-maps of $S$ (endowed with the topology of compact convergence), the identity function is an isolated point \cite{Ab1989}*{Theorem~1.2.19}. In this case one can certainly obtain a result of a similar type to Theorem~\ref{cai} for $S$, by using a normal families argument (we omit the details), but it is of little consequence; the theorem is only significant for families of functions that come arbitrarily close to the identity function. 

If $S$ has an abelian fundamental group, and is not simply connected, then it must be doubly connected; any such Riemann surface is conformally equivalent either to an annulus $A_r=\{z : 1/r<|z|<r\}$ (where $r>1$) or to the punctured disc $\mathbb{D}^*=\mathbb{D}\setminus\{0\}$. The conformal automorphisms of $A_r$ are rotations, and rotations composed with the map $z\longmapsto 1/z$. The remaining holomorphic self-maps of $A_r$ are all homotopic in the family of continuous self-maps of $A_r$ to constant maps \cite{BeMi2007}*{Corollary~13.7}, and the set of these maps does not contain the identity function in its closure. Thus there is no worthwhile analogue of Theorem~\ref{cai} for annuli either, since any sequence of holomorphic self-maps of $A_r$ that converges to the identity function must eventually consist of rotations, and their geometry is straightforward. 

The punctured disc $\mathbb{D}^*$ is different, however, because there are plenty of nontrivial holomorphic self-maps of $\mathbb{D}^*$ in any neighbourhood of the identity function. To consider these maps, we use the universal covering map $\pi\colon\mathbb{H}\longrightarrow\mathbb{D}^*$ given by $\pi(\zeta)=e^{2\pi i \zeta}$, where $\mathbb{H}$ is the upper half-plane. We use $\mathbb{H}$ for the universal covering space rather than $\mathbb{D}$ because doing so gives a simpler covering map (and $\mathbb{H}$ is also marginally easier to work with than $\mathbb{K}$).

Any holomorphic self-map $f$ of $\mathbb{D}^*$ lifts to a holomorphic self-map $\tilde{f}$ of $\mathbb{H}$ with $\pi\circ \tilde{f}=f\circ \pi$. This map $\tilde{f}$ satisfies $\tilde{f}(\zeta+1)=\zeta+m$, for all $\zeta\in \mathbb{H}$ and some nonnegative integer $m$. The integer $m$ is called the \emph{degree} of $f$, and it is denoted by $\degree(f)$. It can also be defined using the formula
\[
\degree(f)=\frac{1}{2\pi i}\int_\gamma \frac{f'(z)}{f(z)}\,dz,
\]
where $\gamma(t)=\tfrac12e^{2\pi it}$ ($t\in [0,1]$). Since $\degree$ is a continuous function, the set of holomorphic self-maps of $\mathbb{D}^*$ of degree $m$, for any nonnegative integer $m$, is a closed set. The identity function belongs to the set of maps of degree 1. For more on the degree, see \cite{BeMi2007}*{Section~13}.

The origin is a removable singularity for any holomorphic self-map $f$ of $\mathbb{D}^*$, and if $\degree(f)>0$, then the origin is fixed by $f$. It is reasonable, therefore, to expect to obtain a one-point inequality for self-maps of $\mathbb{D}^*$ of positive degree akin to the earlier two-point inequalities. The next theorem is of this type; it is similar to Corollary~\ref{xjb}, but the conformal automorphism of $\mathbb{D}$ is replaced by a holomorphic self-covering map of $\mathbb{D}^*$. Such a map has the form $h(z)=e^{i\theta}z^m$, where $\theta\in\mathbb{R}$ and $m\in\mathbb{N}$. In this theorem, $\lambda^*(z)=-1/(|z|\log|z|)$ is the Riemannian density on $\mathbb{D}^*$ that gives rise to the hyperbolic metric $\rho^*$ on $\mathbb{D}^*$.

\begin{theorem}\label{puc}
Suppose that $f$ is a holomorphic self-map of $\mathbb{D}^*$, $h$ is a holomorphic self-covering map of $\mathbb{D}^*$, and $\degree(f)=\degree(h)>0$. Suppose also that $a,z\in\mathbb{D}^*$. Then
\[
\rho^*(f(z),h(z)) \leq L^3\rho^*(f(a),h(a)),
\]
where $L= 8\lambda^*(a)\exp \rho^*(z,a)$.
\end{theorem}

When $h$ is the identity function and $\degree(f)=1$, we obtain a one-point inequality comparable with Theorem~\ref{cai}.

\begin{corollary}
Suppose that $f$ is a holomorphic self-map of $\mathbb{D}^*$ with $\degree(f)=1$ and $a,z\in\mathbb{D}^*$. Then 
\[
\rho^*(f(z),z)\leq L^3 \rho^*(f(a),a),
\]
where $L= 8\lambda^*(a)\exp \rho^*(z,a)$.
\end{corollary}

\textsl{Acknowledgements}. The authors thank the referee for insightful suggestions; in particular, for directing us to a version of the one-point inequality for covering maps. 

\section{Holomorphic maps with a fixed point}

In this section we prove the following theorem, which is a version of our main result, Theorem~\ref{cai}, for holomorphic self-maps of the disc with a fixed point.

\begin{theorem}\label{pkc}
Suppose that $a,b,z\in\mathbb{D}$, with $a\neq b$, and $f$ is a holomorphic self-map of $\mathbb{D}$ that fixes $b$. Then
\[
\rho(f(z),z) \leq M\rho(f(a),a),
\]
where
\[
M= \frac{\exp(\rho(a,z)+\rho(z,b))}{4\sinh\tfrac12\rho(a,b)}.
\]
\end{theorem}

It suffices to prove this theorem when $b=0$, as can be seen by conjugating $f$ by a conformal automorphism of $\mathbb{D}$ that takes $b$ to $0$. So let us assume, henceforth, that $b=0$. We can also assume that $z\neq 0$, because, for $b=0$, the inequality clearly holds when $z=0$.

We recall two formulas for the hyperbolic metric on $\mathbb{D}$ (see  \cite{BeMi2007}*{page~15}), which state that, for $u,v\in\mathbb{D}$,
\begin{equation}\label{uis}
\sinh\tfrac12 \rho(u,v) = \frac{|u-v|}{\sqrt{(1-|u|^2)(1-|v|^2)}},\qquad\cosh\tfrac12 \rho(u,v) = \frac{|1-u\overline{v}|}{\sqrt{(1-|u|^2)(1-|v|^2)}}.
\end{equation}
The equations in the next lemma are merely special cases of these formulas, with $v=0$.

\begin{lemma}\label{iwu}
If $u\in\mathbb{D}$, then
\begin{enumerate}
\item\label{aba}  $\sinh\tfrac12 \rho(u,0) = \dfrac{|u|}{\sqrt{1-|u|^2}}$,
\item\label{abb} $\cosh\tfrac12 \rho(u,0)=\dfrac{1}{\sqrt{1-|u|^2}}$.
\end{enumerate}
\end{lemma}

We can use the equations in this lemma to replace the square-root terms from the left-hand formula from \eqref{uis} in two ways, to give two more formulas involving the hyperbolic metric, presented in the following lemma.

\begin{lemma}\label{fii}
If $u,v\in\mathbb{D}$, then 
\begin{enumerate}
\item\label{uud} $|u-v| = \dfrac{\sinh\tfrac12 \rho(u,v)}{\cosh\tfrac12 \rho(u,0)\cosh\tfrac12 \rho(v,0)}$,
\item\label{uue} $\dfrac{|u-v|}{|u|} = \dfrac{\sinh\tfrac12 \rho(u,v)}{\sinh\tfrac12 \rho(u,0)\cosh\tfrac12 \rho(v,0)}$.
\end{enumerate}
\end{lemma} 

We will apply Lemmas~\ref{iwu} and~\ref{fii} repeatedly, so it is handy to define
\[
s(u,v) = \sinh \tfrac12\rho(u,v) \quad\text{and}\quad c(u,v)=\cosh\tfrac12 \rho(u,v).
\]

Let us proceed with the proof of Theorem~\ref{pkc}. We can  assume that $f$ is not a conformal automorphism of $\mathbb{D}$ fixing the origin (a Euclidean rotation) because such maps are limits of sequences of holomorphic maps that are not conformal automorphisms (in the topology of compact convergence), and the inequality is preserved on taking this type of limit.

We define
\[
g(w) =
\begin{cases}
f(w)/w, & w\neq 0,\\
f'(0), & w=0. 
\end{cases} 
\]
Since $|f(w)|<|w|$ for $w\neq 0$, by  Schwarz's lemma, we see that $g$ is also a holomorphic map from $\mathbb{D}$ to itself.

Recall that $a,z\in\mathbb{D}\setminus\{0\}$. Then $|1-g(z)|\leq |1-g(a)|+|g(a)-g(z)|$. That is, 
\[
\frac{|z-f(z)|}{|z|} \leq \frac{|a-f(a)|}{|a|}+|g(a)-g(z)|.
\]
Applying Lemma~\ref{fii} to this inequality, we obtain
\[
\frac{s(f(z),z)}{s(z,0)c(f(z),0)} \leq \frac{s(f(a),a)}{s(a,0)c(f(a),0)}+\frac{s(g(a),g(z))}{c(g(a),0)c(g(z),0)}.
\]
Since $c(f(z),0)\leq c(z,0)$, by the Schwarz--Pick lemma, and $c(f(a),0)\geq 1$, we can rearrange this inequality to give

\begin{equation}\label{cat}
s(f(z),z) \leq s(z,0)c(z,0)\left(\frac{s(f(a),a)}{s(a,0)}+\frac{s(g(a),g(z))}{c(g(a),0)c(g(z),0)}\right).
\end{equation}
Next, observe that $\rho(g(z),0) \geq |\rho(g(a),0)-\rho(g(a),g(z))|$, so, since cosh is an even function,
\[
c(g(z),0) \geq c(g(a),0)c(g(a),g(z))-s(g(a),0)s(g(a),g(z)).
\]
Multiplying both sides by $c(g(a),0)$ and then applying the equations $c(g(a),0)^2=1/(1-|g(a)|^2)$ and $s(g(a),0)c(g(a),0)=|g(a)|/(1-|g(a)|^2)$ (from Lemma~\ref{iwu}), we see that
\begin{align*}
c(g(z),0)c(g(a),0) & \geq  c(g(a),0)^2c(g(a),g(z))-s(g(a),0)c(g(a),0)s(g(a),g(z))\\
& =  \frac{c(g(a),g(z))-|g(a)|s(g(a),g(z))}{1-|g(a)|^2}\\
&\geq \frac{c(g(a),g(z))-|g(a)|s(g(a),g(z))}{1+|g(a)|} \frac{|a|}{|a-f(a)|}\\
&= \frac{c(g(a),g(z))-|g(a)|s(g(a),g(z))}{1+|g(a)|} \frac{s(a,0)c(f(a),0)}{s(f(a),a)},
\end{align*}
where, in the last line, we have applied Lemma~\ref{fii}\ref{uue} again. Rearranging this, we find that 
\[
\frac{1}{c(g(z),0)c(g(a),0) }\leq \frac{1+|g(a)|}{c(g(a),g(z))-|g(a)|s(g(a),g(z))}\times \frac{s(f(a),a)}{s(a,0)}.
\]
We now combine this inequality with \eqref{cat} to give
\begin{equation}\label{dus}
s(f(z),z) \leq \frac{s(z,0)c(z,0)}{s(a,0)}\left(1+\frac{(1+|g(a)|)s(g(a),g(z))}{c(g(a),g(z))-|g(a)|s(g(a),g(z))}\right)s(f(a),a).
\end{equation}
The part in large brackets is equal to 
\[
\frac{c(g(a),g(z))+s(g(a),g(z))}{c(g(a),g(z))-|g(a)|s(g(a),g(z))} \leq \frac{c(g(a),g(z))+s(g(a),g(z))}{c(g(a),g(z))-s(g(a),g(z))} \leq e^{\rho(a,z)},
\]
where, for the last inequality, we applied the Schwarz--Pick lemma to $g$. Since $s(z,0)c(z,0)=\tfrac12\sinh \rho(z,0) \leq \tfrac14 e^{\rho(z,0)}$, we see that \eqref{dus} reduces to 
\begin{equation}\label{oop}
s(f(z),z) \leq \frac{e^{\rho(a,z)+\rho(z,0)}}{4s(a,0)}s(f(a),a).
\end{equation}
To finish, observe that Theorem~\ref{pkc} is clearly true if $\rho(f(z),z)<\rho(f(a),a)$, because $M>1$. Assume then that $\rho(f(z),z)\geq \rho(f(a),a)$. The function $x\longmapsto \sinh x/x$ is increasing for $x>0$, as one can prove by differentiating it, so 
\[
\frac{\sinh \tfrac12\rho(f(z),z)}{\tfrac12\rho(f(z),z)} \geq \frac{\sinh \tfrac12\rho(f(a),a)}{\tfrac12\rho(f(a),a)}.
\] 
Hence
\[
\frac{\rho(f(z),z)}{\rho(f(a),a)}\leq \frac{s(f(z),z)}{s(f(a),a)}.
\]
This inequality, together with \eqref{oop}, give the inequality of Theorem~\ref{pkc}, completing the proof.

\section{Holomorphic maps of the disc}

This section proves the main result, Theorem~\ref{cai}. In the following preliminary lemma, we refer to a conformal automorphism of $\mathbb{D}$ that is a hyperbolic M\"obius transformation as a \emph{hyperbolic automorphism}. Its \emph{axis} is the hyperbolic line connecting its two fixed points.

\begin{lemma}\label{qlo}
Let $c$ be a point that lies on the axis of a hyperbolic automorphism $h$ of $\mathbb{D}$. Then
\[
\rho(w,h(w))\leq e^{\rho(w,c)}\rho(c,h(c)),
\] 
for all $w\in\mathbb{D}$.
\end{lemma}
\begin{proof}
Let $\gamma$ be the axis of $h$. By \cite{Be1995}*{Theorem~7.35.1}, we have 
\[
\sinh \tfrac{1}{2}\rho (w,h(w))= \cosh \rho(w,\gamma) \sinh\tfrac{1}{2}\rho(c,h(c)),
\]
for $c\in\gamma$ and $w\in\mathbb{D}$. Now, as we mentioned earlier, the function $x\longmapsto \sinh x/x$ is increasing for $x>0$, and $\rho(c,h(c))\leq \rho(w,h(w))$, so
\[
\frac{\rho (w,h(w))}{\rho(c,h(c))}\leq \frac{\sinh \tfrac{1}{2}\rho (w,h(w))}{\sinh\tfrac{1}{2}\rho(c,h(c))}=\cosh \rho(w,\gamma)\leq e^{\rho(w,\gamma)}.
\]
The result then follows from the inequality $\rho(w,\gamma)\leq \rho(w,c)$.
\end{proof}

Now we prove Theorem~\ref{cai}. Suppose, then, that $f$ is a holomorphic self-map of $\mathbb{D}$ and that $a$, $b$ and $z$ are points in $\mathbb{D}$, with $a\neq b$. If $f$ does not fix $b$, then there is a unique hyperbolic line $\gamma$ through $b$ and $f(b)$. Let $h$ be the hyperbolic automorphism of $\mathbb{D}$ with axis $\gamma$ that satisfies $h f(b)=b$. If $f$ fixes $b$, then we define $h$ to be the identity function.  So, applying Theorem~\ref{pkc} to $h f$, we see that 
\[
\rho(h f(z),z) \leq M\rho(h f(a),a),\quad\text{where}\quad  M= \frac{e^{\rho(a,z)+\rho(z,b)}}{4\sinh\tfrac12\rho(a,b)}.
\]
Hence
\begin{align*}
\rho(f(z),z) &\leq \rho(f(z),h f(z))+\rho(h f(z),z) \\
&\leq \rho(f(z),h f(z))+M\rho(h f(a),a) \\
& \leq  \rho(f(z),h f(z))+M\rho(f(a),h f(a)) +M\rho(f(a),a).
\end{align*}
Next, for $u\in\mathbb{D}$, Lemma~\ref{qlo} (with $w=f(u)$ and $c=f(b)$) tells us that if $f(b)\neq b$, then 
\[
\rho(f(u),h f(u)) \leq e^{\rho(f(u),f(b))} \rho(f(b),h f(b)) \leq e^{\rho(u,b)}\rho(f(b),b),
\] 
using the Schwarz--Pick lemma for the final inequality. Clearly, this inequality also holds if $f(b)=b$. Since $e^{\rho(a,b)}>4\sinh\tfrac12\rho(a,b)$, we see that
\begin{align}\label{now}
\rho(f(z),z) &\leq e^{\rho(z,b)}\rho(f(b),b) + Me^{\rho(a,b)}\rho(f(b),b)+M\rho(f(a),a) \notag\\
&\leq (e^{\rho(z,b)}+Me^{\rho(a,b)})(\rho(f(a),a)+\rho(f(b),b)) \notag\\
&\leq \frac{e^{\rho(a,z)+\rho(a,b)+\rho(z,b)}}{2\sinh\tfrac12\rho(a,b)}(\rho(f(a),a)+\rho(f(b),b)).
\end{align}
Theorem~\ref{cai} can be deduced from \eqref{now}, since $\sinh x \geq x$ for all $x>0$. However, we highlight the slightly stronger inequality of \eqref{now} for use later.

\section{Holomorphic maps of the punctured disc}

It remains to prove Theorem~\ref{puc}, and this final section is dedicated to that one task. Recall that  $\lambda^*(z)=-1/(|z|\log |z|)$ is the density  for the hyperbolic metric on the punctured unit disc. We use the following trivial estimates.

\begin{lemma}\label{eea}
If $z\in\mathbb{D}^*$, then 
\begin{enumerate}
\item $\lambda^*(z) \geq e$,
\item $\lambda^*(z)\geq -\log|z|$,
\item $\lambda^*(z)\geq -\dfrac{1}{\log|z|}$.
\end{enumerate}
\end{lemma}

Let us suppose, as stated in Theorem~\ref{puc}, that $f$ is a holomorphic self-map of $\mathbb{D}^*$, $h$ is a holomorphic self-covering map of $\mathbb{D}^*$, and $\degree(f)$ and $\degree(h)$ are both equal to some positive integer $m$. By post-composing $f$ and $h$ with a suitable rotation of $\mathbb{D}^*$ about $0$ (a hyperbolic isometry of $\mathbb{D}^*$) we can assume that $h(z)=z^m$.

Suppose that $a,z\in\mathbb{D}^*$. Let $\pi\colon\mathbb{H}\longrightarrow\mathbb{D}^*$ be the universal covering map $\pi(\zeta)=e^{2\pi i \zeta}$.  We denote the hyperbolic metric on $\mathbb{H}$ by $\rho$. Let $\tilde{a}$ be any point in $\mathbb{H}$ such that $\pi(\tilde{a})=a$. Since 
\[
\rho^*(z,a)=\inf\{\rho(\zeta,\tilde{a}):\zeta \in\mathbb{H}\text{ and }\pi(\zeta)=z\},
\]
 we can choose $\tilde{z}\in\mathbb{H}$ such that $\pi(\tilde{z})=z$ and $\rho(\tilde{z},\tilde{a})=\rho^*(z,a)$.

We now lift the map $f$ to a holomorphic map $\tilde{f}\colon\mathbb{H}\longrightarrow\mathbb{H}$ with $\pi \circ \tilde{f}=f\circ \pi$. The map $\tilde{f}$ satisfies $\tilde{f}(\zeta+1)=\tilde{f}(\zeta)+m$, for all $\zeta\in \mathbb{H}$, since $f$ has degree $m$. We  also lift the map $h$ to the holomorphic map $\tilde{h}(\zeta)=m\zeta$, which is a hyperbolic isometry of $\mathbb{H}$. Observe that $\pi(\tilde{f}(\tilde{a}))=f(a)$ and $\pi(\tilde{h}(\tilde{a}))=h(a)$. By replacing $\tilde{f}$ with a map that is the composition of $\tilde{f}$ followed by a suitable integer translation (also a lift of $f$), we can assume that $\rho(\tilde{f}(\tilde{a}),\tilde{h}(\tilde{a}))= \rho^*(f(a),h(a))$. 

Next we apply the slightly stronger version of Theorem~\ref{cai} given by inequality \eqref{now} to the function $\tilde{h}^{-1}\circ\tilde{f}$ and the points $\tilde{a}$, $\tilde{a}+1$ and $\tilde{z}$. We obtain
\[
\rho(\tilde{h}^{-1}\circ\tilde{f}(\tilde{z}),\tilde{z})\leq K\big(\rho(\tilde{h}^{-1}\circ\tilde{f}(\tilde{a}),\tilde{a})+\rho(\tilde{h}^{-1}\circ\tilde{f}(\tilde{a}+1),\tilde{a}+1)\big),
\]
where $K=e^{\rho(\tilde{z},\tilde{a})+\rho(\tilde{a},\tilde{a}+1)+\rho(\tilde{a}+1,\tilde{z})}/(2\sinh\tfrac12\rho(\tilde{a},\tilde{a}+1))$. Since $\tilde{h}$ is a hyperbolic isometry of $\mathbb{H}$, we see that
\begin{align*}
\rho(\tilde{f}(\tilde{z}),\tilde{h}(\tilde{z})) &\leq K\big(\rho(\tilde{f}(\tilde{a}),\tilde{h}(\tilde{a}))+\rho(\tilde{f}(\tilde{a}+1),\tilde{h}(\tilde{a}+1))\big)\\
&= K\big(\rho(\tilde{f}(\tilde{a}),\tilde{h}(\tilde{a}))+\rho(\tilde{f}(\tilde{a})+m,\tilde{h}(\tilde{a})+m)\big)\\
&= 2K\rho(\tilde{f}(\tilde{a}),\tilde{h}(\tilde{a})).
\end{align*}
Now
\begin{align*}
K &= \frac{\exp\left(\rho(\tilde{z},\tilde{a})+\rho(\tilde{a},\tilde{a}+1)+\rho(\tilde{a}+1,\tilde{z})\right)}{2\sinh\tfrac12\rho(\tilde{a},\tilde{a}+1)}\\
&\leq \frac{\exp\left(2\rho(\tilde{z},\tilde{a})+2\rho(\tilde{a},\tilde{a}+1)\right)}{2\sinh\tfrac12\rho(\tilde{a},\tilde{a}+1)}\\
&\leq 2e^{2\rho(\tilde{z},\tilde{a})}\frac{\cosh^2\rho(\tilde{a},\tilde{a}+1)}{\sinh\tfrac12\rho(\tilde{a},\tilde{a}+1)}.
\end{align*}
We can write this expression in Euclidean terms by using the following standard formulas for the hyperbolic metric on $\mathbb{H}$, taken from \cite{BeMi2007}*{Theorem~7.4}:
\[
\cosh\rho(u,v) = 1 +\frac{|u-v|^2}{2\Im u\Im v},\qquad \sinh \tfrac12\rho(u,v)  = \frac{|u-v|}{2\sqrt{\Im u \Im v}},
\]
where $u,v\in\mathbb{H}$. Thus
\[
K \leq 2e^{2\rho(\tilde{z},\tilde{a})}\frac{(1+\tfrac12(\Im\tilde{a})^{-2})^2}{\tfrac12(\Im\tilde{a})^{-1}}= e^{2\rho(\tilde{z},\tilde{a})}(4\Im\tilde{a}+4(\Im\tilde{a})^{-1}+(\Im\tilde{a})^{-3}).
\]
Now $e^{2\pi i \tilde{a}}=a$, so  $e^{-2\pi\Im\tilde{a}}=|a|$. Hence  $\Im\tilde{a}= -(\log |a|)/(2\pi)$, and we can apply Lemma~\ref{eea} to deduce that $\Im\tilde{a}\leq \lambda^*(a)/(2\pi)$ and $(\Im\tilde{a})^{-1}\leq 2\pi\lambda^*(a)$. Since $\lambda^*(a)\geq e$, we obtain
\[
4\Im\tilde{a}+4(\Im\tilde{a})^{-1}+(\Im\tilde{a})^{-3} \leq (1+8\pi)\lambda^*(a)+(2\pi)^3\lambda^*(a)^3\leq 252\lambda^*(a)^3.
\]
Therefore
\[
\rho(\tilde{f}(\tilde{z}),\tilde{h}(\tilde{z})) \leq 504\lambda^*(a)^3 e^{2\rho(\tilde{z},\tilde{a})}\rho(\tilde{f}(\tilde{a}),\tilde{h}(\tilde{a})).
\]
The proof of Theorem~\ref{puc} is complete on observing that 
\[
\rho^*(f(z),h(z))\leq \rho(\tilde{f}(\tilde{z}),\tilde{h}(\tilde{z})), \quad \rho^*(f(a),h(a))=\rho(\tilde{f}(\tilde{a}),\tilde{h}(\tilde{a}))\quad\text{and}\quad\rho^*(z,a)=\rho(\tilde{z},\tilde{a}).
\]

\end{document}